\documentclass[11pt]{article}
\usepackage{amsmath,amsthm,amsfonts,amssymb,amscd, amsxtra}
\usepackage[notref,notcite]{showkeys}
\theoremstyle{plain}
\newtheorem{theorem}{Theorem}[section]
\newtheorem{lemma}{Lemma}[section]
\newtheorem{definition}{Definition}[section]
\newtheorem{corollary}{Corollary}[section]
\newtheorem{proposition}{Proposition}[section]

\newtheorem{remark}{Remark}[section]

\newcommand{\beq}{\begin{equation}}
\newcommand{\eeq}{\end{equation}}
\newcommand{\beqa}{\begin{eqnarray}}
\newcommand{\eeqa}{\end{eqnarray}}
\newcommand{\beqas}{\begin{eqnarray*}}
\newcommand{\eeqas}{\end{eqnarray*}}

\def\argmin{\operatorname{argmin}}

\def\min{\operatorname{min}}
\def\max{\operatorname{max}}
\DeclareMathOperator{\grad}{grad}

\begin{document}

\title{Local Convergence of the Proximal Point Method for a Special Class of Nonconvex Functions on Hadamard Manifolds}

\author{
G. C. Bento\thanks{IME, Universidade Federal de Goi\'as,
Goi\^ania, GO 74001-970, BR (Email: {\tt glaydston@mat.ufg.br})}
\and
O. P. Ferreira
\thanks{IME, Universidade Federal de Goi\'as,
Goi\^ania, GO 74001-970, BR (Email: {\tt orizon@mat.ufg.br}). The
author was supported in part by CNPq Grant 302618/2005-8,
PRONEX--Optimization(FAPERJ/CNPq) and FUNAPE/UFG.}\and P. R.
Oliveira \thanks{COPPE-Sistemas, Universidade Federal do Rio de
Janeiro, Rio de Janeiro, RJ 21945-970, BR (Email: {\tt
poliveir@cos.ufrj.br}). This author was supported in part by
CNPq.} } \date{April 10, 2010}

\maketitle
\begin{abstract}
Local convergence analysis of the proximal point method for special class of nonconvex function
on Hadamard  manifold is presented in this paper. The well
definedness of the sequence generated by the proximal point method
is guaranteed. Moreover, is proved that each  cluster point of
this sequence satisfies the necessary optimality conditions and,
under additional assumptions, its convergence for a minimizer is
obtained.
\end{abstract}
\section{Introduction}
The extension of the concepts and techniques of the Mathematical
Programming of  the Euclidean space $R^{n}$ to Riemannian
manifolds is natural. It has been frequently done in recent years,
with a theoretical purpose and also to obtain effective
algorithms; see \cite{Absil2007}, \cite{Teboulle2004},
\cite{FO98}, \cite{FO2000}, \cite{Zhu2007}, \cite{PO2009},
\cite{RAP97}, \cite{S94} and \cite{U94}. In particular, we observe
that, these extensions allow the solving some nonconvex
constrained problems in Euclidean space. More precisely, nonconvex
problems in the classic sense may become convex with the
introduction of an adequate Riemannian metric on the manifold
(see, for example \cite{XFLN2006}). The proximal point algorithm,
introduced by Martinet \cite{M70} and Rockafellar \cite{R76}, has
been extended to different contexts, see \cite{FO2000},
\cite{PO2009} and their references. In \cite{FO2000} the authors
generalized the proximal point method for solve convex
optimization problems of the form
\begin{eqnarray}\label{po:conv}
\begin{array}{clc}
  (P) & \min  f(p) \\
   & \textnormal{s.t.}\,\,\, p\in M,\\
\end{array}
\end{eqnarray}
where $M$ is a Hadamard manifold and $f:M\to\mathbb{R}$ is a
convex  function (in the Riemannian sense). The method was described
as follows:
\begin{equation}\label{E:1.2}
p^{k+1}:=\argmin_{p\in M} \left\{f(p)+\frac{\lambda_k}{2}d^2(p,{p^k})\right\},
\end{equation}
with $p^{\circ}\in M$ an arbitrary point, $d$ the intrinsic
Riemannian  distance (to be defined later on) and $\{\lambda_k\}$
a sequence of positive numbers. The authors also showed that this
extension is natural. As regards to \cite{PO2009} the authors
generalized the proximal point method with Bregman distance to
solve quasiconvex and convex optimization problems also on
Hadamard manifold. Spingarn in \cite{Spingarn1982} has, in
particular, developed the proximal point method for the
minimization of a certain class of nondifferentiable noncovex
functions, namely, the lower-$C^2$ functions defined in Euclidean
spaces, see also  \cite{Sagastizabal2005}. Kaplan and Tichatschke
in \cite{Kaplan1998} also applied the proximal point method for
the minimization of a similar class of the ones of
\cite{Sagastizabal2005} and \cite{Spingarn1982}, namely, functions defined as maximum of a certain collection (finite/infinite) of continuously differentiable functions. In \cite{GFO2008} we study, in the Riemannian context, the same class of functions studied  in \cite{Kaplan1998}. In that context we applied the proximal point method \eqref{E:1.2} to solve the problem \eqref{po:conv}, however we assumed that the collection of functions defining the objective function was finite.

Our goal is to extend the results of \cite{GFO2008}. We consider that the objective function is given by the maximum of a collection infinite of continuously differentiable functions. To obtain the results in \cite{GFO2008}, it was necessary to study the generalized directional derivative in the Riemannian manifolds context. In this paper we go further in the study of properties of the generalized directional derivative in order to analyze the convergence of the  proximal point method. Several works have studied such concepts and presented many useful results in the Riemannian optimization context, see for example \cite{Azagra2005}, \cite{Zhu2007}, \cite{Montreanu-Pavel1982} and \cite{Thãmelt1992}.

The paper is divided as follows. In Section \ref{sec2} we give
the notation and some results on the Riemannian geometry which we
will use along the paper. In Section \ref{sec3} we recall some
facts of the convex analysis on Hadamard manifolds. In Section
\ref{sec:dd} we present definition of generalized
directional derivative of a locally Lipschitz function (not
necessarily convex) which, in the Euclidean case,  coincides with the Clarke's generalized directional derivative. Moreover, some properties of that derivative are presented, amongst which the upper semicontinuity of the directional derivative. In Section \ref{sec5} we study the proximal point method (\ref{E:1.2}) to solve the problem \eqref{po:conv}, in the case where the objective function is a real-valued function (non necessarily convex) on a Hadamard manifold $M$ given by the maximum of a certain class of
functions. Finally in Section \ref{sec6} we provide an
example where the proximal point method for nonconvex problems is applied.

\subsection{Notation and terminology} \label{sec2}
In this section we introduce some fundamental  properties and
notations on Riemannian geometry. These basics facts can be
found in any introductory book on Riemannian geometry, such as in
\cite{MP92} and \cite{S96}.

Let $M$ be a $n$-dimentional connected manifold. We denote by
$T_pM$  the $n$-dimentional {\it tangent space} of $M$ at $p$, by
$TM=\cup_{p\in M}T_pM$ {\itshape{tangent bundle}} of $M$ and by
${\cal X}(M)$ the space of smooth vector fields over $M$. When $M$
is endowed with a Riemannian metric $\langle \,,\, \rangle$, with
the corresponding norm denoted by $\| \; \|$, then $M$ is now a
Riemannian manifold. Recall that the metric can be used to define
the lenght of piecewise smooth curves $\gamma:[a,b]\rightarrow M$
joining $p$ to $q$, i.e., such that $\gamma(a)=p$ and
$\gamma(b)=q$, by
\[
l(\gamma)=\int_a^b\|\gamma^{\prime}(t)\|dt,
\]
and, moreover, by minimizing this length functional over the set
of all such curves, we obtain a Riemannian distance $d(p,q)$ which
induces the original topology on $M$. The metric induces a map
$f\mapsto\grad f\in{\cal X}(M)$ which associates to each smooth function 
on $M$ its gradient via the rule $\langle\grad
f,X\rangle=d f(X),\ X\in{\cal X}(M)$. Let $\nabla$ be the
Levi-Civita connection associated to $(M,{\langle} \,,\,
{\rangle})$. A vector field $V$ along $\gamma$ is said to be {\it
parallel} if $\nabla_{\gamma^{\prime}} V=0$. If $\gamma^{\prime}$
itself is parallel we say that $\gamma$ is a {\it geodesic}. Given
that geodesic equation $\nabla_{\ \gamma^{\prime}}
\gamma^{\prime}=0$ is a second order nonlinear ordinary
differential equation, then geodesic $\gamma=\gamma _{v}(.,p)$ is
determined by its position $p$ and velocity $v$ at $p$. It is easy
to check that $\|\gamma ^{\prime}\|$ is constant. We say that $
\gamma $ is {\it normalized} if $\| \gamma ^{\prime}\|=1$. The
restriction of a geodesic to a  closed bounded interval is called
a {\it geodesic segment}. A geodesic segment joining $p$ to $q$ in
$ M$ is said to be {\it minimal} if its length equals $d(p,q)$ and
this geodesic is called a {\it minimizing geodesic}. If $\gamma$
is a curve joining points $p$ and $q$ in $ M$ then, for each $t\in
[a,b]$, $\nabla$ induces a linear isometry, relative to ${
\langle}\, ,\, {\rangle}$,
$P_{\gamma(a)\gamma(t)}:T_{\gamma(a)}M\to T_{\gamma(t)}M$, the
so-called {\it parallel transport} along $\gamma$ from $\gamma(a)$
to $\gamma(t)$. The inverse map of $P_{\gamma(a)\gamma(t)}$ is
denoted by $P_{\gamma(a)\gamma(t)}^{-1}:T_{\gamma(t)}  M \to
T_{\gamma(a)}M$. In the particular case  of $\gamma$ is the unique
curve joining points $p$ and $q$ in $M$ then parallel transport
along $\gamma$ from $p$ to $q$ is denoted by $P_{pq}:T_{p}M\to
T_{q}M$.

A Riemannian manifold is {\it complete} if geodesics are defined
for any values of $t$. Hopf-Rinow's theorem asserts that if this
is the case then any pair of points, say $p$ and $q$, in $M$ can
be joined by a (not necessarily unique) minimal geodesic segment.
Moreover, $( M, d)$ is a complete metric space and bounded and
closed subsets are compact. Take $p\in M$. The {\it exponential
map} $exp_{p}:T_{p}  M \to M $ is defined by  $exp_{p}v\,=\,
\gamma _{v}(1,p)$.

We denote by $R$ {\it the curvature tensor \/} defined by
$R(X,Y)=\nabla_{X}\nabla_{Y}Z-
\nabla_{Y}\nabla_{X}Z-\nabla_{[Y,X]}Z$,  with $X,Y,Z\in{\cal
X}(M)$, where $[X,Y]=YX-XY$. Then the {\it sectional curvature \/}
with respect to $X$ and $Y$ is given by $K(X,Y)=\langle R(X,Y)Y ,
X\rangle /(||X||^{2}||X||^{2}-\langle X\,,\,Y\rangle ^{2})$, where
$||X||=\langle X,X\rangle ^{1/2}$. If $K(X,Y)\leqslant 0$ for all
$X$ and $Y$, then $M$ is called  a {\it Riemannian manifold of
nonpositive curvature \/} and we use the short notation
$K\leqslant 0$.
\begin{theorem} \label{T:Hadamard}
Let $M$ be a complete, simply connected Riemannian manifold with nonpositive sectional
curvature. Then $M$ is diffeomorphic to the Euclidean space $\mathbb{R}^n $, $ n=dim M $. More
precisely, at any point $p\in M $, the exponential mapping $ exp_{p}:T_{p}M \to M $ is a diffeomorphism.
\end{theorem}
\begin{proof}
See \cite{MP92} and \cite{S96}.
\end{proof}

A complete simply connected Riemannian manifold of nonpositive
sectional   curvature  is called a {\it{Hadamard manifold}}. The
Theorem \ref{T:Hadamard} says that if $M$ is Hadamard manifold,
then $M$ has the same topology and differential structure of the
Euclidean space $\mathbb{R}^n$. Furthermore, are known some
similar geometrical properties of the Euclidean space
$\mathbb{R}^n$, such as, given two points there exists an unique
geodesic that joins them. {\it In this paper, all manifolds $M$
are assumed to be Hadamard finite dimensional}.

\section{Convexity in Hadamard manifold} \label{sec3}
In this section, we introduce some fundamental properties and
notations  of convex analysis on Hadamard manifolds which will be
used later. We will see that these properties are similar to those
obtained in convex analysis on the Euclidean space
$\mathbb{R}^{n}$. References to convex analysis on Euclidean
space $\mathbb{R}^{n}$ are in \cite{HL93}, and on Riemannian manifold
are in \cite{XFL2002}, \cite{FO2000}, \cite{RAP97},
\cite{S96}, \cite{S94} and \cite{U94}.

The set $\Omega\subset M$ is said to be {\it convex \/} if any
geodesic segment with end points in $\Omega$ is contained in
$\Omega$. Let $\Omega\subset M$ be an open convex set. A function
$f:M\to\mathbb{R}$ is said to be {\it convex} (respectively, {\it
strictly convex}) on $\Omega$ if for any geodesic segment
$\gamma:[a, b]\to\Omega$ the composition $f\circ\gamma:[a,
b]\to\mathbb{R}$ is convex (respectively, strictly convex). Now, a
function $f:M\to\mathbb{R}$ is said to be {\it strongly convex} on
$\Omega$ with constant $L>0$ if, for any geodesic segment
$\gamma:[a, b]\to\Omega$, the composition $f\circ\gamma:[a,
b]\to\mathbb{R}$ is strongly convex with constant
$L\|\gamma'(0)\|^2$. Take $p\in M$. A vector $s \in T_pM$ is said
to be a {\it subgradient\/} of $f$ at $p$ if
\[
f(q) \geq f(p) + \langle s,\exp^{-1}_pq\rangle,
\]
for any $q\in M$. The set of all subgradients of $f$ at $p$,
denoted  by $\partial f(p)$, is called the {\it subdifferential\/}
of $f$ at $p$.

Take ${p}\in M $. Let  $ exp^{-1}_{p}:M\to T_{p}M$ be the  inverse
of the  exponential map which is also $C^{\infty }$.  Note that
$d({q}\, , \, p)\,=\,||exp^{-1}_{p}q||$, the map
$d^2(\,.\,,{p})\colon M\to\mathbb{R}$ is  $C^{\infty}$ and
\[
\grad \frac{1}{2}d^2(q,{p})=-exp^{-1}_{q}{p},
\]
(remember that  $M$ is a Hadamard manifold) see, for example, \cite{S96}.
\begin{proposition}\label{FunDistConv}
Take ${p}\in M$. The map $d^2(\,.\,,{p})/2$ is strongly convex.
\end{proposition}
\begin{proof}
See \cite{XFL2002}.
\end{proof}
\begin{definition}\label{def2.14}
Let $\Omega\subset M$ be an open convex set. A function $f: M \to
\mathbb{R}$ is said  to be Lipschitz on $\Omega$ if there exists a
constant $L:=L(\Omega)\geq 0$ such that
\begin{equation}\label{Lipsch1}
|f(p)-f(q)|\leq Ld(p,q), \qquad p,q\in\Omega.
\end{equation}
Moreover, if it is established that for all $p_0\in \Omega$ there
exists $L(p_0)\geq 0$ and $\delta=\delta(p_0)>0$ such that the
inequality (\ref{Lipsch1}) occurs with $L=L(p_0)$ for all $p,q\in
B_{\delta}(p_0):=\{p\in \Omega : d(p,p_0)<\delta\}$, then $f$ is
called locally Lipschitz on $\Omega$.
\end{definition}
\begin{remark}\label{obs2.15}
As an immediate consequence of the triangular inequality we obtain
that $|d(p,p_0)-d(q,p_0)|\leq d(p,q)$ for all $p,q$ and $p_0\in M$.
Then, of the Definition \ref{def2.14}, we get that the Riemannian
distance function to a  fixed point, $d(\cdot,q)$ is  Lipschitzian
and therefore Lipschitzian locally. In fact, it well known that every
convex function is locally Lipschitz and consequently continuous.
See \cite{GW73}.
\end{remark}
\begin{proposition}\label{SubClarke2}
Let $\Omega \subset M$ be a open convex set, $f:M
\to\mathbb{R}$ and $p\in M$. If  there exists $\lambda >0$ such
that  $f+(\lambda/2)\, d^{2}( .\,, p):M \to\mathbb{R}$ is convex
on $\Omega$, then $f$ is Lipschitz locally on $\Omega$.
\end{proposition}
\begin{proof}
Because $f+(\lambda/2)\, d^{2}( .\,, p)$ is convex, it follows
from  Remark \ref{obs2.15}  that for any $\tilde{p}\in\Omega$
there exist $L_1,\delta_1>0$ such that
\begin{equation}\label{DesLip100}
\left|[f(q_1)+(\lambda/2)\, d^{2}( q_1\,, p)]
-[f(q_2)+(\lambda/2)\, d^{2}( q_2\,, p)]\right|\leq L_1
d(q_1,q_2), \qquad \forall\; q_1,q_2\in B(\tilde{p},\delta_1).
\end{equation}
Moreover, Proposition \ref{FunDistConv} together with Remark
\ref{obs2.15} imply that there exist $L_2,\delta_2>0$ such that
\begin{equation}\label{DesLip101}
|(1/2)d^2( q_1\,, p)-(1/2)d^2( q_2\,, p)|\leq L_2d(q_1,q_2), \qquad \forall\; q_1,q_2\in B(\tilde{p},\delta_1).
\end{equation}
Simples algebraic manipulations implies that
\begin{multline*}
|f(q_1)-f(q_2)|\leq \left|[f(q_1)+\lambda/2)\, d^{2}( q_1\,,
p)]-[f(q_2)+(\lambda/2)\, d^{2}( q_2\,, p)]\right|+\\
+\left|(\lambda/2)\, d^{2}( q_2\,, p)-(\lambda/2)\, d^{2}( q_1\,,
p)\right|.
\end{multline*}
Therefore, taking $\delta=\min\{\delta_1,\delta_2\}$,   using
\eqref{DesLip100} and \eqref{DesLip101} we conclude from last
inequality that
\[
|f(q_1)-f(q_2)|\leq (L_1+\lambda L_2)d(q_1,q_2), \qquad \forall\; q_1,q_2\in B(\tilde{p},\delta),
\]
and the proof is finished.
\end{proof}
\begin{definition}
Let $\Omega\subset M$ be a open convex set and $f:M\to\mathbb{R}$
a  continuously differentiable function on $\Omega$. The gradient
vector  field $\grad f$ is said to be Lipschitz with constant
$\Gamma \geq 0$ on $\Omega$ always that
\[
\|\grad f(q)-P_{pq}\grad f(p)\|\leq\Gamma
d(p,q), \qquad p, q \in \Omega,
\]
where  $P_{pq}$ is the parallel transport along the geodesic segment joining $p$ to $q$.
\end{definition}
\section{Generalized directional derivatives}\label{sec:dd}
In this section we present definitions for the generalized
directional derivative and subdifferential of a locally Lipschitz function (not
necessarily convex) which, in the Euclidean case, coincide with the Clarke's generalized directional derivative and subdifferential, respectively. Moreover, some properties of those concepts are presented, amongst them the upper semicontinuity of the directional derivative and a relationship  between  the subdifferential of a sum of two Lipschitz locally function  (in the particular case that one of them is differentiable) and its  subdifferentials.
\begin{definition} \label{d:Clarke}
Let $\Omega\subset M$ be an open convex set and $f:M\to\mathbb{R}$
a locally Lipschitz function on $\Omega$. The generalized
directional derivative $f^\circ:T\Omega \to \mathbb{R}$ of  $f$ is
defined by
\begin{equation}\label{Clarke1}
f^\circ(p,v):=\limsup\limits_{t\downarrow 0\ q\to p}\frac{f\left(\exp_q t(D\exp_p)_{\exp^{-1}_pq}v\right)-f(q)}{t},
\end{equation}
where $(D\exp_p)_{\exp^{-1}_pq}$ denotes the differential of $\exp_p$ at $\exp^{-1}_pq$.
\end{definition}
It is worth to pointed out that an equivalently definition has appeared in \cite{Azagra2005}.
\begin{remark}
The generalized directional derivative is well defined. Indeed,
let $L_p>0$ the Lipschitz  constant of $f$ at $p$ and $\delta=\delta(p)>0$
such that
\[
|f(\exp_q t(D\exp_p)_{\exp^{-1}_pq}v)-f(q)|\leq L_p \,d(\exp_q t(D\exp_p)_{\exp^{-1}_pq}v, \,q), \quad q\in B_{\delta}(p), \quad t\in[0,\delta).
\]
Since $d(\exp_q t(D\exp_p)_{\exp^{-1}_pq}v, \,q)= t \| (D\exp_p)_{\exp^{-1}_pq}v\|$, above inequality becomes
\[
|f(\exp_q t(D\exp_p)_{\exp^{-1}_pq}v)-f(q)|\leq L_p \, t \| (D\exp_p)_{\exp^{-1}_pq}v\|, \quad q\in B_{\delta}(p), \quad t\in[0,\delta).
\]
Since $\lim_{ q\to p}\,(D\exp_p)_{\exp^{-1}_pq}v=v$ our statement follows from last  inequality.
\end{remark}
\begin{remark}
Note that, if $M=\mathbb{R}^n$ then  $\exp_{p}w=p+w$ and
$
(D\exp_p)_{\exp^{-1}_pq}v=v.
$
In this case, \eqref{Clarke1} becomes
\[
f^{\circ}(p,v)=\limsup\limits_{t\downarrow 0\ q\to p}\frac{f(q+tv)-f(q)}t,
\]
which is the Clarke's generalized directional derivative, see
\cite{Clarke1983}.  Therefore, the generalized differential
derivative on Hadamard manifold  is a natural extension of the
Clarke's generalized differential derivative.
\end{remark}
Now we are going to prove the upper semicontinuity of the
generalized  directional derivative.
\begin{proposition}\label{PropoConverg}
Let $\Omega\subset M$ be an open convex set and $f:M\to\mathbb{R}$
be  a locally Lipschitz  function. Then, $f^{\circ}$ is upper
semicontinuous on $T\Omega$, i.e., if $(p,v)\in T\Omega$   and
$\{p^k,v^k\}$ is a sequence in $T\Omega$ such that
$\lim_{k\to+\infty}(p^k,v^k)=(p,v)$, then
\begin{equation}\label{DerFinita2}
\limsup_{k\to+\infty}f^\circ(p^k,v^k)\leq f^\circ(p,v).
\end{equation}
\end{proposition}
\begin{proof}
Let $(p,v)\in T\Omega$ and $\{(p^k,v^k)\}\subset T\Omega$ such that $\lim_{k\to+\infty}(p^k,v^k)=(p,v)$.
For proving the inequality \eqref{DerFinita2} first  note that for each $k$
\[
f^\circ(p^k,v^k)\leq \limsup_{t\downarrow 0\ (q,w)\to (p^k,v^k)}\frac{f(\exp_qtw)-f(q)}t,\qquad (q,w)\in T\Omega.
\]
So, by  definition of upper limit, there exists  $(q^k,w^k)\in T\Omega-\{(p^k,v^k)\}$ and $t_k>0$ such that
\begin{equation}\label{UpperSem1}
f^\circ(p^k,v^k)-\frac{1}{k}< \frac{f(\exp_{q^k}t_kw^k)-f(q^k)}{t_k},
\qquad \tilde{d}((q^k,w^k),(p^k,v^k))+t_k<\frac{1}{k},
\end{equation}
with $\tilde{d}$ being the Riemannian distance in $TM$. Let
$U_p\subset \Omega$ be a  neighborhood of $p$ such that
$TU_p\approx U_p\times\mathbb{R}^n$, $f$ is Lipschitz in
$U_p$ with constant $L_p$ and $\exp$ is Lipschitz on $TU_p$ with constant $K$. From the first inequality in
\eqref{UpperSem1}, we obtain
\begin{multline}\label{UpperSem3}
f^\circ(p^k,v^k)-\frac{1}{k}< \, \frac{f\left(\exp_{q^k}t_k(D\exp_p)_{\exp^{-1}_pq^k}v\right)-f(q^k)}{t_k}\, +\\ \frac{f(\exp_{q^k}t_kw^k)-f\left(\exp_{q^k}t_k(D\exp_p)_{\exp^{-1}_pq^k}v\right)}{t_k}.
\end{multline}
On the other hand, as $\lim_{k\to+\infty}(p^k,v^k)=(p,v)$,  we
conclude from the second inequality in \eqref{UpperSem1} that
\[
\exp_{q^k}t_kw^k\in U_p, \qquad  \exp_{q^k}t_k(D\exp_p)_{\exp^{-1}_pq^k}v)\in U_p, \qquad k>k_0,
\]
 for $k_0$ sufficiently large. Thus, as $f$ is Lipschitz on $U_p$,  for $ k>k_0$ we have
\begin{multline}\label{UpperSemI}
\left|f(\exp_{q^k}t_kw^k)-f\left(\exp_{q^k}t_k(D\exp_p)_{\exp^{-1}_pq^k}v\right)\right|\leq\\
L_p\,
d\left(\exp_{q^k}t_kw^k,\,\exp_{q^k}t_k(D\exp_p)_{\exp^{-1}_pq^k}v\right).
\end{multline}
Now, taking into account that $\exp$ is Lipschitz on $TU_p$, in the particular case that $k>k_0$ 
\[
d\left(\exp_{q^k}t_kw^k,\,\exp_{q^k}t_k(D\exp_p)_{\exp^{-1}_pq^k}v\right)\leq k\|t_kw^k-t_k(D\exp_p)_{\exp^{-1}_pq^k}v\|.
\]
Since $\lim_{k\to+\infty}p^k=p$,  second equation in
\eqref{UpperSem1} imply that $\lim_{k\to+\infty}q^k=p$. Consequently,
$$\lim_{ k\to
+\infty}\,(D\exp_p)_{\exp^{-1}_pq_k}v=v,$$  which,  together with last inequality imply
\[
\lim_{k\to+\infty}d\left(\exp_{q^k}t_kw^k,\,\exp_{q^k}t_k(D\exp_p)_{\exp^{-1}_pq^k}v\right)/t_k=0.
\]
Therefore, combining last equation, \eqref{UpperSem3},
\eqref{UpperSemI},  and Definition \ref{d:Clarke} the result
follows.
\end{proof}
Next we generalize the  definition of subdifferential for locally Lipschitz functions defined on Hadamard manifold.
\begin{definition}
Let $\Omega\subset M$ be an open convex set and
$f:M\to\mathbb{R}$ a locally Lipschitz function on $\Omega$. The
generalized subdifferential of $f$ at $p\in \Omega$, denoted by
$\partial^{\circ} f(p)$, is defined by
\[
\partial^{\circ}f(p):=\big{\{}w\in T_pM : f^{\circ}(p,v)\geq\langle w,v\rangle, \forall \; v\in T_pM \big{\}}.
\]
\end{definition}
\begin{remark}\label{regular 1}
Let $\Omega\subset M$ be an open convex set. If the function
$f:M\to\mathbb{R}$ is convex on $\Omega$, then
$f^{\circ}(p,v)=f'(p,v)$  (respectively,  $\partial^\circ
f(p)=\partial f(p)$) for all $p\in \Omega$, i.e., the directional
derivatives (respectively,  subdifferential) for Lipschitz
functions is a generalization of the directional derivatives
(respectively,  subdifferential) for convex functions. See
\cite{Azagra2005} Claim $5.4$ in the proof of Theorem $5.3$.
\end{remark}

\begin{definition}
Let $f:M\to\mathbb{R}$ be locally Lipschitz function. A
point $p\in \Omega$ is said to be a stationary point of $f$ always  $0\in
\partial^{\circ}f(p)$.
\end{definition}
\begin{lemma}\label{lemacl1}
Let $\Omega\subset M$ be an open set. If $f:M\to\mathbb{R}$ is locally Lipschitz function on $\Omega$ and $g:M\to\mathbb{R}$ is convex on $\Omega$, then
\begin{equation}\label{ddg:1}
(f+g)^{\circ}(p,v)=f^{\circ}(p,v)+g'(p,v)\qquad p\in \Omega, \quad v\in T_pM.
\end{equation}
Moreover, if $g$ is differentiable, we have
\begin{equation}\label{sdg:1}
\partial^{\circ}(f+g)(p)=\partial^{\circ}f(p)+ \grad g(p),\qquad p\in\Omega.
\end{equation}
\end{lemma}
\begin{proof}
Using the definition of the generalized directional derivative and simple algebraic manipulations, we obtain
\begin{multline*}
(f+g)^{\circ}(p,v)=\limsup\limits_{t\downarrow 0\ q\to p}
\left[\frac{f(\exp_qt(D\exp_p)_{\exp^{-1}_pq}v)-f(q)}{t}+\right.\\ \left.\frac{g(\exp_qt(D\exp_p)_{\exp^{-1}_pq}v)-g(q)}{t}\right].
\end{multline*}
From basic properties of the upper limit along with the definition of directional derivative generalized and Remark~ \ref{regular 1}, follows that
\begin{equation}\label{P:ddg1}
(f+g)^{\circ}(p,v)\leq f^{\circ}(p,v)+g'(p,v).
\end{equation}
Now, as $f^{\circ}(p,v)=\left((f+g)+(-g)\right)^{\circ}(p,v)$, above inequality implies, in particular
\[
f^{\circ}(p,v)\leq (f+g)^{\circ}(p,v)+(-g)'(p,v),
\]
which is equivalent to
\[
(f+g)^{\circ}(p,v)\geq f^{\circ}(p,v)+g'(p,v).
\]
Thus, combining last inequality with inequality \eqref{P:ddg1}, the equality \eqref{ddg:1} is obtained.

In the case that $g$ is also differentiable, in particular $g'(p,v)=\langle\grad g(p),v\rangle$. Therefore, the proof of the equality \eqref{sdg:1} is an immediate consequence of the equality \eqref{ddg:1} along with the definition of the generalized subdifferential.
\end{proof}

\begin{corollary} \label{c:ssfd}
Let $\Omega \subset M$ be a open convex set, $f:M\to\mathbb{R}$ be locally Lipschitz functions on $\Omega$, $\tilde{p}\in M$ and $\lambda >0$ such
that  $f+(\lambda/2)\, d^{2}( .\,, \tilde{p}):M \to\mathbb{R}$ is convex
on $\Omega$. If  $p\in \Omega$ is a minimizer of $f+(\lambda/2)d^2(.,\tilde{p})$ then
\[
\lambda \exp^{-1}_p\tilde{p}\in \partial^{\circ} f(p).
\]
\end{corollary}
\begin{proof}
Since $p$ is a minimizer of $f+(\lambda/2)d^2(.,\tilde{p})$ we obtain
\begin{equation}\label{eq:mc1}
0\in\partial \left(f+\frac{\lambda}{2}d^2(.\, ,\,\tilde{p})\right)(p).
\end{equation}
On the other hand, as $f+(\lambda/2)\, d^{2}( .\,, \tilde{p})$ is convex
on $\Omega$ and $(\lambda/2)\, d^{2}( .\,, \tilde{p})$ is differentiable with $\grad \left(\lambda/2\right)d^2(q,{p})=-\lambda\exp^{-1}_{q}{p}$, using Remark~\ref{regular 1}, Proposition~\ref{SubClarke2} and applying Lema~ \ref{lemacl1} with $g=(\lambda/2)\, d^{2}( .\,, \tilde{p})$, we have
\begin{equation}\label{eq:mc2}
\partial\left(f+\frac{\lambda}{2}d^{2}(.,\,\tilde{p})\right)(p)=\partial^{\circ}\left(f+\frac{\lambda}{2}d^{2}(.,\, \tilde{p})\right)(p)=\partial^{\circ} f(p)-\lambda \exp^{-1}_p\tilde{p}.
\end{equation}
Therefore, the result follows by combining \eqref{eq:mc1} with \eqref{eq:mc2}.
\end{proof}

\section{Proximal Point Method for Nonconvex Problems}\label{sec5}
In this section we present an application of the proximal point
method for minimize a real-valued function (non necessarily
convex) given by the maximum of a certain class of  continuously
differentiable functions. Our goal is to prove the following
theorem:
\begin{theorem}\label{MPP10}
Let $\Omega\subset M$ be an open convex set, $q\in M$ and
$T\subset\mathbb{R}$ a  compact set. Let $\varphi:M\times T\to
\mathbb{R}$ be a continuous function on $\Omega\times T$ such that $\varphi(.,\tau):M\to
\mathbb{R}$ is a continuously differentiable function on $\bar{\Omega}$ (closure of $\Omega$)
for all $\tau\in T$, and $f:M\to\mathbb{R}$ defined by
\[
f(p):=\max_{\tau\in T} \varphi(p,\tau).
\]
Assume that  $-\infty<\inf_{p\in M}f(p)$, $\grad_{p}
\varphi(.,\tau)$ is Lipschitz on  $\Omega$ with constant
$L_{\tau}$ for each $\tau\in T$ such that $\sup_{\tau\in
T}L_{\tau}<+\infty$ and
\[
L_f(f(q))=\left\{p\in M: f(p)\leq f(q)\right\}\subset \Omega,\qquad \inf_{p\in M}f(p)<f(q).
\]
Take $0<\bar{\lambda}$ and a sequence $\{ \lambda_{k}\}$
satisfying $\sup_{\tau\in T}L_{\tau}<\lambda_k\leq \bar{\lambda}$
and $\hat{p}\in L_f(f(q))$. Then the proximal point method
\begin{equation}\label{E:1.22}
p^{k+1}:=\argmin_{p\in M} \left\{f(p)+\frac{\lambda_k}{2}d^2(p,{p^k})\right\}, \qquad k=0, 1, \ldots,
\end{equation}
with starting point $p^0=\hat{p}$ is well defined, the generated
sequence  $\{p^{k}\}$ rests in $L_f(f(q))$ and satisfies only one
of the following statement
\begin{itemize}
\item[i)] $\{p^{k}\}$ is finite, i.e., $p^{k+1}=p^k$ for some $k$
and, in this case, $p^k$ is a stationary point of $f$, \item[ii)]
$\{p^{k}\}$ is infinite and, in this case, any cluster point of
$\{p^k\}$ is a stationary point of $f$.
\end{itemize}
Moreover, assume that  the minimizer set of $f$ is non-empty, i. e.,
\begin{itemize}
\item[{\bf h1)}] $U^*=\{p : f(p)=\inf_{p\in M}f(p)\}\neq\emptyset$.
\end{itemize}
Let $c\in(\inf_{p\in M}f(p),\; f(q))$. If, in addition, the
following assumptions hold:
\begin{itemize}
\item[{\bf h2)}] $L_f(c)$ is convex, $f$ is convex on $L_f(c)$ ; 
\item[{\bf h3)}] for all $p\in L_f(f(q))\setminus L_f(c)$ and $y(p)\in \partial^\circ f(p)$ we have $\|y(p)\|>\delta>0$,
\end{itemize}
then the sequence $\{p^{k}\}$ generated by \eqref{E:1.22} with
\begin{equation}\label{mpp102}
\sup_{\tau\in T}L_\tau<\lambda_k\leq\bar{\lambda}, \qquad k=0, 1, \ldots
\end{equation}
converge to a point $p^*\in U^*$.
\end{theorem}
\begin{remark}
The continuity of each function $\varphi(.,\tau)$ on
$\bar{\Omega}$ in {\bf h2} guarantees that the level sets
of the fuction f, in particular the solution set $U^*$, are closed in the topology of the manifold $M$.
\end{remark}

In the next remark we show that if $\Omega$ is bounded and
$\varphi(.,\tau)$  is convex on  $\Omega$ for all $\tau\in T$ then $f$ satisfies the
assumptions {\bf h2} and {\bf h3}. 
\begin{remark}\label{re:tconver}
If $\varphi(.,\tau)$ is a convex function on $\Omega$ for all $\tau\in T$ then  the
assumtion {\bf h2} is naturally verified and if {\bf h1} hold then
{\bf h3} also holds. For details, see \cite{GFO2008}. 
\end{remark}

In order to prove above theorem we need of some preliminary results.
From now on we assume  that every assumptions on Theorem \ref{MPP10}
hold,  with the exception of {\bf h1}, {\bf h2} and {\bf h3},
which will be considered to hold only when explicitly stated.
\begin{lemma}\label{MPP8}
For all $\tilde{p}\in M$ and $\lambda$ satisfying
\[
\sup_{\tau\in T}L_{\tau}<\lambda ,
\]
function  $f+(\lambda/2) d^2(. \,, \tilde{p} )$ is strongly
convex on $\Omega$ with  constant $\lambda - \sup_{\tau \in
T}L_{\tau}$.
\end{lemma}
\begin{proof}
Since $T$ is compact and $\varphi$ is continuous the well
definition of $f$ follows. To conclude, see Lemma 4.1 in
\cite{GFO2008}.
\end{proof}
\begin{corollary} \label{cor:wdf}
The proximal point method \eqref{E:1.22} applied to $f$ with starting point $p^0=\hat p$ is well defined.
\end{corollary}
\begin{proof}
Since compactness play no rule,  the proof is equal to the proof of Corollary 4.1 in \cite{GFO2008}.
\end{proof}
\begin{lemma}\label{mpprox10}
Let $\{p^k\}$ be the sequence generated by the proximal point method \eqref{E:1.22}. Then
\begin{itemize}
\item[i)]
$
0\in\partial \left(f+\frac{\lambda_k}{2}d^2(.\, ,\,{p^k})\right)(p^{k+1}), \quad k=0, 1, \ldots.
$
\item[ii)] $\lim\limits_{s\to\infty}d(p^{k+1},p^{k})=0.$
\end{itemize}Moreover, if $\lambda_k$ satisfies \eqref{mpp102} and {\bf h1}, {\bf h2} and {\bf h3} hold, then $\{p^k\}$ converges  to a  point $p^*\in U^*$.
\end{lemma}
\begin{proof}
Since compactness play no rule, the proof is similar to the proof of Lemmas 4.2, 4.3 and 4.4 of \cite{GFO2008}.
\end{proof}
\subsubsection*{Proof of {\bf Theorem \ref{MPP10}}}
\begin{proof}
The well definition of the proximal point method \eqref{E:1.22} follows from  the
Corollary~\ref{cor:wdf}. Let  $\{p^{k}\}$ be the sequence
generated by proximal point method. Because $p^0=\hat{p}\in L_f(f(q))$, \eqref{E:1.22} implies that the whole sequence is in $L_f(f(q))$. From item i of Lemma~\ref{mpprox10}, we have
\[
0\in\partial \left(f+\frac{\lambda_k}{2}d^2(.\, ,\,{p^k})\right)(p^{k+1}), \quad k=0, 1, \ldots.
\]
Since $\sup_{\tau\in T}L_{\tau}<\lambda_k$, Lemma \ref{MPP8}
implies  that $f+(\lambda_k/2)d^2(.\,,{p^k})$ is strongly convex
on $\Omega$, which together with Proposition \ref{SubClarke2} give
us that $f$ is locally Lipschitz on $\Omega$. So, using the definition of $p^{k+1}$, we conclude from  Corollary \ref{c:ssfd} with $\lambda=\lambda_k$, $\tilde{p}=p^k$ and $p=p^{k+1}$ that
\begin{equation}\label{MPP12}
\lambda_k\exp^{-1}_{p^{k+1}} p^k\in\partial^{\circ}f(p^{k+1}).
\end{equation}
If $\{p^{k}\}$ is finite, then $p^{k+1}=p^k$ for some $k$ and
latter  inclusion implies that $0\in\partial^{\circ}f(p^{k+1})$,
i.e., $p^{k}$ is a stationary point of $f$. Now assume that
$\{p^k\}$ is a infinite sequence. If $\bar{p}$ is a cluster
point of $\{p^k\}$, then there exists a subsequence $\{p^{k_s}\}$ of  $\{p^{k}\}$ such
that $\lim_{s\to+\infty}p^{k_s+1}=\bar{p}$ and item~ii of Lemma~\ref{mpprox10} implies
\begin{equation}\label{AcumEstac1}
\lim\limits_{s\to\infty}\|\exp^{-1}_{p^{k_s+1}} p^{k_s}\|=\lim\limits_{s\to\infty}d(p^{k_s+1},p^{k_s})=0.
\end{equation}
Now, from the relation \eqref{MPP12}, we have
\[
f^\circ(p^{k_s+1},v)\geq \lambda_{k_s}\langle\exp^{-1}_{p^{k_s+1}} p^{k_s}, v\rangle,\qquad \forall\; v\in T_{p^{k_s+1}}M.
\]
Let $\bar{v}\in T_{\bar{p}}M$. Hence, latter inequality implies that
\[
f^\circ(p^{k_s+1},v^{k_s+1})\geq \lambda_{k_s}\langle\exp^{-1}_{p^{k_s+1}} p^{k_s}, v^{k_s+1}\rangle,\qquad
v^{k_s+1}=D(\exp_{\bar{p}})_{\exp^{-1}_{\bar{p}}p^{k_s+1}}\bar{v}.
\]
Note that $\lim_{s\to+\infty}p^{k_s+1}=\bar{p}$ implies
$\lim_{s\to+\infty}v^{k_s+1}=\bar{v}$.  Because
$\{\lambda_{k_s}\}$ is bounded, letting $s$ goes to $+\infty$ in
the last inequality, Proposition \ref{PropoConverg} together with
\eqref{AcumEstac1} give us
\[
f^\circ(\bar{p}, \bar{v})\geq\lim\limits_{s\to+\infty}\sup f^\circ(p^{k_s+1},v^{k_s+1})\geq 0,
\]
which implies that
$
0\in\partial^{\circ} f(\bar{p}),
$
i.e., $\bar{p}$ is a stationary point of $f$ and the first part of the theorem is concluded.

The second part follows from the last part of Lemma \ref{mpprox10} and the proof of the theorem is finished.
\end{proof}
\section{Example}\label{sec6}
Let $(\mathbb{R}_{++}, \langle \, , \, \rangle)$ be the Riemannian manifold, where $\mathbb{R}_{++}=\{x\in\mathbb{R}:x>0\}$ and $\langle \, , \, \rangle$ is the Riemannian metric   $\langle u , v \rangle=g(x)uv$ with $g:\mathbb{R}_{++}\to (0,+\infty)$. So,  the Christoffel symbol and the geodesic equation are given by
\[
\Gamma(x)=\frac{1}{2}g^{-1}(x)\frac{dg(x)}{dx}=\frac{d}{dx}\ln\sqrt{g(x)}, \qquad \frac{d^2x}{dt^2}+\Gamma(x)\left( \frac{dx}{dt}\right)^2=0,
\]
respectively. Besides, in relation to the twice differentiable function $h:\mathbb{R}_{++}\to\mathbb{R}$, the Gradient and the Hessian of $h$  are given by
\[
\grad h=g^{-1}h',  \qquad {\rm hess}\ h=h''-\Gamma h',
\]
respectively, where $h'$ and $h''$ denote  the first and second derivatives of $h$ in the Euclidean sense. For more details  see \cite{U94}. 
So, in the particular case of $g(x)=x^{-2}$, 
\begin{equation}\label{Hess:1}
\Gamma(x)=-x^{-1},\quad \grad h(x)=x^2h'(x),\quad {\rm hess}\ h(x)=h''(x)+x^{-1}h'(x).
\end{equation}
Moreover, the map $\psi :\mathbb{R} \to \mathbb{R}_{++}$ defined  by $\psi(x)={\rm e}^x$ is 
an isometry between the Euclidean space $\mathbb{R}$ and the manifold $(\mathbb{R}_{++}, \langle \, , \, \rangle)$, and the Riemannian distance $d:\mathbb{R}_{++}\times\mathbb{R_{++}}\to\mathbb{R}_{+}$ is given by 
\begin{equation}\label{dRiem:1}
d(x,y)=|\psi^{-1}(x)-\psi^{-1}(y)|=|\ln(x/y)|,
\end{equation}
see, for example \cite{XFLN2006}. Therefore,  $(\mathbb{R}_{++}, \langle \, , \, \rangle)$ is a Hadamard manifold and the unique geodesic $x:\mathbb{R}\to\mathbb{R}_{++}$ with initial conditions $x(0)=x_0$ and $x'(0)=v$ is given by
\[
x(t)=x_0 {\rm e}^{(v/x_0)t}.
\]
Now let $f_1,f_2,f:\mathbb{R}_{++}\to\mathbb{R}$ and $\varphi:\mathbb{R}_{++}\times[0,1]\to\mathbb{R}$ be real-valued functions such that 
\begin{equation} \label{eq:dfs}
\varphi(x,\tau)=f_1(x)+t(f_2(x)-f_1(x)),\quad f(x)=\max_{\tau\in[0,1]}\varphi(x,\tau),
\end{equation}
and consider the problem 
\[ 
\begin{array}{clc}
   & \min f(x) \\
   & \textnormal{s.t.}\,\,\, x\in\mathbb{R}_{++}.\\
\end{array}
\]
Take a sequence $\{\lambda_{k}\}$ satisfying $0<\lambda_k$. From \eqref{dRiem:1}, the proximal point method \eqref{E:1.22} becomes
\[
x^{k+1}:=\argmin_{x\in\mathbb{R_{++}}}\left\{f(x)+\frac{\lambda_k}{2}\ln^2\left(\frac{x}{x^k}\right)\right\}, \qquad k=0, 1, \ldots. 
\]
If $f_1$ and $f_2$ are given, respectively, by $f_1(x)=\ln(x)$ and $f_2(x)=-\ln(x)+{\rm e}^{-2x}-{\rm e}^{-2}$, then $\varphi$ is continuous and $\varphi(.,\tau)$ is continuously differentiable for each $\tau\in [0,1]$. The last expression in \eqref{Hess:1} implies that
\begin{equation}\label{ExHess:1}
{\rm hess}\ f_1(x)=0\quad, \quad {\rm hess}\ f_2(x)=(4-2/x){\rm e}^{-2x}, \qquad x\in \mathbb{R}_{++},
\end{equation}
e, as a consequence, first expression in \eqref{eq:dfs} give us
\[
{\rm hess}_x\ \varphi(x,\tau)=\tau{\rm hess}\ f_2(x),\quad \forall\, x\in\mathbb{R}_{++}\quad\forall\,\tau\in[0,1].
\]
Note that, for $0<\epsilon< 1/4$ and $\Omega=(\epsilon,+\infty)$, ${\rm hess}\ f_2$ is bounded on $\Omega$ and therefore $\grad f_2$ is Lipschitz on $\Omega$. We denote by $L$ the constant of Lipschitz of $\grad f_2$. From the last equality ${\rm hess}_x\ \varphi(.,\tau)$ is also bounded on $\Omega$ and $\grad_x\varphi(.,\tau)$ is Lipschitz on $\Omega$ with constant $L_{\tau}=\tau L$ for all $\tau\in[0,1]$. Besides, $\sup_{\tau\in[0,1]}L_{\tau}=L<+\infty$.

We claim that $f(x)=\max_{j=1,2}f_j(x)$. Indeed, note that $f_2(x)-f_1(x)>0$ for $x\in(0,1)$, $f_2(x)-f_1(x)<0$ for $x\in(1,+\infty)$ and $f_1(1)=f_2(1)$. Thus the affine  function  $[0, 1]\ni\tau \mapsto \varphi(x,\tau)$ satisfies
\[
\max_{\tau\in[0,1]}\varphi(x,\tau)=
\begin{cases}
f_1(x), \qquad x\in(0,1), \\
f_2(x), \qquad x\in(1,+\infty).
\end{cases}
\]
and the claim follows. With that characterization for $f$ all assumptions of Theorem \ref{MPP10} are verified, with $q=5/16$, $c=f(3/4)$ and $\delta=2/5$,
see  Example in \cite{GFO2008}.
Hence, letting $x^0\in \mathbb{R}_{++}$ and $\bar{\lambda}>0 $ such that $x^0\in L_f(f(q))$ and $L<\mu<\lambda_k\leq\bar{\lambda}$, the proximal point method, characterized in Theorem~\ref{MPP10}, can be applied for solving the above nonconvex problem.  

\begin{remark}
Function $f(x)=\max_{\tau}\varphi(x,\tau)$, in the above example, is  nonconvex  (in the Euclidean sense)  when restricted to any open neighborhood containing its minimizer $x^*=1$. Therefore, the local classical proximal point method (see \cite{Kaplan1998}) cannot be applied to minimize that function. Also, as $f$ is nonconvex in the Riemannian sense, the Riemannian proximal point method (see \cite{FO2000}) can not be applied to minimize that function, see Example in \cite{GFO2008} for more details.
\end{remark}

\section{Final Remarks}
We have extended the application of the proximal point method to
solve nonconvex optimization  problems on Hadamard manifold in the
case that the objective function is given by the maximum of a
certain infinite collection of continuously differentiable functions. Convexity
of the auxiliary problems is guaranteed with the choice
appropriate regularization parameters in relation to the constants
of Lipschitz of the field gradients of the functions which they
compose the class in subject. With regards to the Theorem
\ref{MPP10}, in the particular case that $\varphi(.,\tau)$ is convex on $\Omega$ for $\tau\in
T$, convexity of the auxiliary problems is guaranteed without need
of restrictive assumptions on the regularization parameters.
Besides, as observed in Remark \ref{re:tconver}, the additional
assumptions {\bf h2} and {\bf h3} are satisfied whenever $\Omega$ is
bounded.


\end{document}